\documentclass[letterpaper]{article}

\newcommand{\Rneginf}{\ensuremath{\mathbb{R}_{-\infty}}}
\newcommand{\Rposinf}{\ensuremath{\mathbb{R}_{\infty}}}
\newcommand{\R}{\ensuremath{\mathbb{R}}}
\newcommand{\Bsl}{\ensuremath{\mathcal{B}^j_{\text{slack}}}}
\newcommand{\Bsu}{\ensuremath{\mathcal{B}^j_{\text{surplus}}}}
\newcommand{\Bijsl}{\ensuremath{\mathcal{B}^{ij}_{\text{slack}}}}
\newcommand{\Bijsu}{\ensuremath{\mathcal{B}^{ij}_{\text{surplus}}}}
\newcommand{\lw}{\ensuremath{\overline{\ell}}}
\newcommand{\uw}{\ensuremath{\overline{u}}}
\newcommand{\code}[1]{\texttt{#1}}
\newcommand{\lnew}{\ell^{\text{new}}}
\newcommand{\unew}{u^{\text{new}}}
\newcommand{\ntotal}{\ensuremath{n^{\text{total}}}}
\newcommand{\ncurrent}{\ensuremath{n^{\text{current}}}}
\newcommand{\Pfinite}{\ensuremath{\mathcal{P}^{\text{fin}}}}
\newcommand{\Pinfinite}{\ensuremath{\mathcal{P}^{\text{inf}}}}
\newcommand{\Prog}{\ensuremath{\mathcal{P}}\xspace}

\newcommand{\seq}{\textsl{seq\_prop}\xspace}
\newcommand{\gpu}{\textsl{gpu\_prop}\xspace}

\usepackage[colorinlistoftodos,prependcaption,textsize=scriptsize]{todonotes}

\usepackage{regexpatch}
\makeatletter
\xpatchcmd{\@todo}{\setkeys{todonotes}{#1}}{\setkeys{todonotes}{inline,color=green!15!yellow,#1}}{}{}
\makeatother

\usepackage{url}
\usepackage{csquotes}
\renewcommand{\mkbegdispquote}[2]{\itshape}
\usepackage{xspace}
\usepackage[group-separator={,}]{siunitx}
\usepackage{booktabs}
\usepackage{amsmath,amssymb,amsfonts,amsthm}

\newtheorem{observation}{Observation}
\newtheorem{corollary}{Corollary}

\usepackage[noend]{algorithmic}
\usepackage{algorithm}
\algsetup{indent=3ex}

\def\BibTeX{{\rm B\kern-.05em{\sc i\kern-.025em b}\kern-.08em
    T\kern-.1667em\lower.7ex\hbox{E}\kern-.125emX}}

\newtheorem{definition}{Definition}

\begin{document}

\begin{center}
 {\LARGE An Algorithm-Independent Measure of Progress for Linear Constraint Propagation}
\end{center}

\vspace{8mm}

\textbf{Boro Sofranac}\hfill{\ttfamily sofranac@zib.de}\\
{\small\emph{Berlin Institute of Technology and Zuse Institute Berlin}}

\textbf{Ambros Gleixner}\hfill{\ttfamily gleixner@zib.de}\\
{\small\emph{HTW Berlin and Zuse Institute Berlin}}

\textbf{Sebastian Pokutta}\hfill{\ttfamily pokutta@zib.de}\\
{\small\emph{Berlin Institute of Technology and Zuse Institute Berlin}}

\vspace{3mm}

\begin{abstract}
  Propagation of linear constraints has become a crucial sub-routine in
modern Mixed-Integer Programming (MIP) solvers.
In practice, iterative algorithms with tolerance-based stopping criteria are
used to avoid problems with slow or infinite convergence.
However, these heuristic stopping criteria can pose difficulties for fairly
comparing the efficiency of different implementations of iterative
propagation algorithms in a real-world setting.
Most significantly, the presence of unbounded variable domains in the problem 
formulation makes it difficult to quantify the relative size of reductions performed on them.
In this work, we develop a method to measure---independently of the algorithmic
design---the progress that a given iterative propagation procedure has
made at a given point in time during its execution.
Our measure makes it possible to study and better compare the behavior of bounds
propagation algorithms for linear constraints.
We apply the new measure to answer two questions of practical relevance:
(i) We investigate to what extent heuristic stopping criteria can lead to
premature termination on real-world MIP instances.
(ii) We compare a GPU-parallel propagation algorithm against a sequential
state-of-the-art implementation and show that the parallel version is even more
competitive in a real-world setting than originally reported.
\end{abstract}

\section{Introduction}
This paper is concerned with \emph{Mixed-Integer Linear Programs} (MIPs) of the form
\begin{equation}
\label{eq:MIP}
\min \{c^Tx \; | \; Ax \leq b, \ell \le x \le u, x \in \mathbb{R}^n , x_j \in \mathbb{Z} \; \text{for all} \; j \in I\},
\end{equation}
\noindent
where $A \in \R^{m \times n}$, $b \in \R^m$, $c \in \R^n$, and $I \subseteq \mathbb{N} = \{1,\ldots, n\}$. Additionally, $\ell \in \Rneginf^n$ and $u \in \Rposinf^n$, where $\Rposinf:=\mathbb{R} \cup \{\infty\}$ and $\Rneginf:=\mathbb{R} \cup \{-\infty\}$. For each variable $x_j$, the interval $[\ell_j, u_j]$ is called its \emph{domain}, which is defined by its lower and upper \emph{bounds} $\ell_j$ and $u_j$, which may be infinite.

Surprisingly fast solvers for solving MIPs have been developed in practice despite MIPs being $\mathcal{NP}$-hard in the worst case \cite{AchterbergWunderling2013,KochMartinPfetsch2013}. To this end, the most successful method has been the \emph{branch-and-bound} algorithm \cite{10.2307/1910129} and its numerous extensions. The key idea of this method is to split the original problem into several sub-problems (\emph{branching}) which are hopefully easier to solve. By doing this recursively, a \emph{search tree} is created with nodes being the individual sub-problems. The \emph{bounding} step solves relaxations of sub-problems to obtain a lower bound on their solutions. This bound can then be used to prune sub-optimal nodes which cannot lead to improving solutions. By doing this, the algorithm tries to avoid having to enumerate exponentially many sub-problems. The most common way to obtain a relaxation of a sub-problem is to drop the integrality constraints of the variables. This yields a \emph{Linear Program} (LP) which can be solved e.g., by the simplex method \cite{NemhauserWolsey1988}.

This core idea is extended by numerous techniques to speed up the solution process. One of the most important techniques is called \emph{constraint propagation}. It improves the formulation of the (sub)problem by removing parts of domains of each variable that it detects cannot lead to \emph{feasible} solutions \cite{cp_handbook}. The more descriptive term \emph{bounds propagation} or \emph{bounds tightening} is used to denote the variants that maintain a continuous interval as domain. Modern MIP solvers make use of this technique during \emph{presolving} in order to improve the global problem formulation \cite{Savelsbergh1994}, as well as during the branch-and-bound algorithm to improve the formulation of the sub-problems at the nodes of the search tree \cite{Achterberg2009}.

In practice, efficient implementations exist in MIP solvers \cite{doi:10.1287/ijoc.2018.0857,Achterberg2009} and recently even a GPU-parallel algorithm \cite{SofranacGleixnerPokutta2020} has been developed. These are iterative methods, which may converge to the tightest bounds only at infinity. For such methods, the presence of unbounded variable domains in the problem formulation makes the quantification of the relative distance to the final result at a given iteration difficult. (Iterative bounds tightening has a unique fixed point to which it converges, see Section \ref{sec:dom_prop_background}.) In turn, this makes it difficult to define an implementation-independent measure of how much progress these algorithms have achieved at a given iteration.

In this paper, we address this difficulty and introduce tools to study and compare the behavior of iterative bounds tightening algorithms in MIP.
We show that the reduction of infinite bounds to some finite values is a fundamentally different process from the subsequent (finite) improvements thereafter, and thus propose to measure the ability of an algorithm to make progress in each of the processes independently. We show how the challenge posed by infinite starting bounds can be solved and provide methods for measuring the progress of both the infinite and the finite domain reductions. Pseudocode and hints are provided to aid independent implementation of our procedure. Additionally, the code of our own implementation is made publicly available.

On the applications side, the new procedure is used to investigate two questions.
First, we analyze to what extent heuristic, tolerance-based stopping criteria as typically imposed by real-world MIP solvers can cause iterative bounds tightening algorithms to terminate prematurely; we find that this situation occurs rarely in practice.
Second, we compare a newly developed, GPU-based propagation algorithm~\cite{SofranacGleixnerPokutta2020} to a state-of-the-art sequential implementation in a real-world setting where both are terminated early; we show that the GPU-parallel version is even more competitive than originally reported.

The rest of the paper is organized as follows. After presenting the necessary background and motivation in Section~\ref{sec:background_and_motivation}, we discuss the properties of bounds propagation and its ability to perform reductions on infinite and on finite bounds in Section~\ref{sec:finite_and_infinite_domain_reductions}. Based on the findings, we present functions used to measure the progress of bounds tightening algorithms in Section~\ref{sec:measuring_function}. Lastly, in Section~\ref{sec:applications}, we apply the developed procedure to answer the above-mentioned questions and present our computational results. Section~\ref{sec:outlook} gives a brief outlook.

\section{Background and Motivation}
\label{sec:background_and_motivation}

In Section \ref{sub:cp_and_mip}, we introduce some basic terminology used in the Constraint Programming (CP) and MIP communities, related to constraint propagation. Section \ref{sec:dom_prop_background} formally presents bounds propagation of linear constraints alongside some known results from literature that are relevant for the discussions in the paper. In Section \ref{sec:motivation} we outline the problems that motivate the paper.

\subsection{Constraint Propagation in CP and MIP}
\label{sub:cp_and_mip}

In the Constraint Programming (CP) community, constraint propagation appears in a variety of forms, both in terms of the algorithms and its desired goals \cite{cp_handbook}.
The propagation algorithms are implemented via mappings called \emph{propagators}. A propagator is a monotonically decreasing function from variable domains to variable domains \cite{eff_prop_engines}. The goal of most propagation algorithms is fomalized through the notion of \emph{consistency}, which these algorithms strive to achieve. The most successful consistency technique is \emph{arc consistency} \cite{MACKWORTH197799}. Multivariate extension of arc consistency has been called \emph{generalized arc consistency} \cite{masini}, as well as \emph{domain consistency} \cite{VANHENTENRYCK1998139}, and \emph{hyper-arc consistency} \cite{marriot}. Informally speaking, a given domain is \emph{domain consistent} for a given constraint if it is the least domain containing all solutions to the constraint (see \cite{cp_handbook} for a formal definition).

The main idea of \emph{bounds consistency} is to relax the consistency requirement to only require the lower and the upper bounds of domains of each variable to fulfill it. There are several bounds consistency notions in the CP literature \cite{choi}. In this paper, we adopt the notion of bounds consistency from \cite[Definition 2.7]{Achterberg2009}.

Modern CP solvers often work with a number of propagators which might or might not strive for different levels of consistency \cite{eff_prop_engines}. In this setting, the notions such as \emph{greatest common fixed point} (see \cite[Definition 4]{intbdpropcomplexity}) and consistency of a system of constraints are often analysed as a product of a set of propagators.  Solvers often focus on optimizing the interplay between different propagators (e.g., see \cite{eff_prop_engines}) to quickly decide feasibility.

In MIP solving, constraint propagation additionally interacts with many other
components that are mostly focused on reaching and proving optimality,
see~\cite{Achterberg_2009_2,scil,simpl,branch_and_infer} for examples of different approaches to integrate constraint
propagation and MIP.
As a result, the role of constraint propagation in the larger solving process
changes and developers are faced with different computational trade-offs.
In practice, propagation is almost always terminated before the fixed point is
reached \cite{Achterberg2009}.
In this paper, we are concerned with constraint propagation of a set of linear
constraints, where we explicitly include the presence of continuous variables
and of variables with initially unbounded domains, which frequently occur in
real-world MIP formulations.

\subsection{Bounds Propagation of Linear Constraints}  
\label{sec:dom_prop_background}
A \emph{linear constraint} can be written in the form
\begin{equation}
\label{eq:lincons}
\underline{\beta} \le \sum_{i=1}^n a_ix_i \le \overline{\beta},
\end{equation}
where $\underline{\beta} \in \Rneginf$ and $\overline{\beta} \in \Rposinf$ are left and right hand sides, respectively, and $a \in \mathbb{R}^n$ is the vector of constraint coefficients. Variables $x_i$ have lower and upper bounds $\ell_i \in \Rneginf$ and $u_i \in \Rposinf$, respectively.\footnote{When $x \in \mathbb{Z}$, then $\ell \in \mathbb{Z}_{-\infty}$ and $u \in \mathbb{Z}_{\infty}$, however, because $\mathbb{Z} \subset \mathbb{R}$, integer variables can be handled the same way as real ones. In the remainder of the paper, $\mathbb{Z}$ will be used only where necessary.} We require the following definitions:
\begin{definition}[activity bounds and residuals]
\label{actsdefinition}
Given a constraint of the form~\eqref{eq:lincons} and bounds $\ell \le x \le u$, the functions $\underline{\alpha}: \Rneginf^n, \Rposinf^n \mapsto \mathbb{R} \cup \{-\infty, \infty\}$ and $\overline{\alpha}: \Rneginf^n, \Rposinf^n \mapsto \mathbb{R} \cup \{-\infty, \infty\}$ are called the \emph{minimum} and \emph{maximum activities} of the constraint, respectively, and are defined as
\begin{subequations}
\begin{align}
\label{eq:minactivities}
& \underline{\alpha} = \underline{\alpha}(\ell,u) = \sum_{i=1}^n a_ib_i \text{ with } b_i = 
  \begin{cases}
      \ell_i & \text{if}\ a_i > 0 \\
      u_i & \text{if}\ a_i < 0
    \end{cases}, 
  \intertext{and }
\label{eq:maxactivities}
 & \overline{\alpha} = \overline{\alpha}(\ell,u) = \sum_{i=1}^n a_ib_i \text{ with } b_i = 
  \begin{cases}
      u_i & \text{if}\ a_i > 0 \\
      \ell_i & \text{if}\ a_i < 0
    \end{cases}.
\end{align}
\end{subequations}
The functions $\underline{\alpha}_j: \Rneginf^n, \Rposinf^n, \{1,\ldots,n\} \mapsto \mathbb{R} \cup \{-\infty, \infty\}$ and $\overline{\alpha}_j: \Rneginf^n, \Rposinf^n, \{1,\ldots,n\} \mapsto \mathbb{R} \cup \{-\infty, \infty\}$ are called the \emph{j-th minimum activity residual} and the \emph{j-th maximum activity residual} of the constraint, and are defined as
\begin{subequations}
\begin{align}
\label{eq:minresactivities}
& \underline{\alpha}_j = \underline{\alpha}_j(\ell,u,j) =\sum_{i=1, i \ne j}^n a_ib_i \text{ with } b_i = 
  \begin{cases}
      \ell_i & \text{if}\ a_i > 0 \\
      u_i & \text{if}\ a_i < 0
    \end{cases}, 
\intertext{and}
\label{eq:maxresactivities}
 & \overline{\alpha}_j = \overline{\alpha}_j(\ell,u,j) = \sum_{i=1, i \ne j}^n a_ib_i \text{ with } b_i = 
  \begin{cases}
      u_i & \text{if}\ a_i > 0 \\
      \ell_i & \text{if}\ a_i < 0
    \end{cases}.
\end{align}
\end{subequations}
\end{definition} 
\begin{definition}[bound candidate functions]
The functions $\Bsu: \Rneginf^n, \Rposinf^n \mapsto \mathbb{R} \cup \{-\infty, \infty\}$ and $\Bsl: \Rneginf^n, \Rposinf^n \mapsto \mathbb{R} \cup \{-\infty, \infty\}$ are called the \emph{bound candidate functions} and are defined as
\begin{subequations}
\begin{align}
& \Bsu(\ell, u) = \frac{\overline{\beta} - \underline{\alpha}_j}{a_j},
\intertext{and}
 & \Bsl(\ell, u) = \frac{\underline{\beta} - \overline{\alpha}_j}{a_j}.
\end{align}
\end{subequations}
\end{definition}

Then, the following observations are true and can be be translated into algorithmic steps, see, e.g., \cite{Achterberg2009,Harvey2003,intbdpropcomplexity}:
\begin{observation}[linear constraint propagation]
\label{obs:dom_prop_steps}
\begin{enumerate}
\item\label{dp:stepi} If $\underline{\beta} \le \underline{\alpha}$ and $\overline{\alpha} \le \overline{\beta}$, then the constraint is redundant and can be removed.
\item\label{dp:stepii} If $\underline{\alpha}>\overline{\beta}$ or $\underline{\beta}>\overline{\alpha}$, then the constraint cannot be satisfied and hence the entire (sub)problem is infeasible.
\item\label{dp:stepiii} Let $x$ satisfy \eqref{eq:lincons}, i.e., $\underline{\beta} \le \sum_{i=1}^n a_ix_i \le \overline{\beta}$, then for all $j=\{1,\ldots,n\}$ with $a_j > 0$,
\begin{subequations}
\begin{align}
\label{eq:bound_candidates_pos}
& \lnew = \Bsl(\ell, u) \le x_j \le \Bsu(\ell, u) = \unew,
\intertext{and for all $j=\{1,\ldots,n\}$ with $a_j < 0$,}
\label{eq:bound_candidates_neg}
& \lnew = \Bsu(\ell, u) \le x_j \le \Bsl(\ell, u) = \unew.
\end{align}
\end{subequations}

\item\label{dp:stepiv} For all $j \in \{1,\ldots,n\}$ such that $x_j \in \mathbb{Z}$,
\begin{equation}
\label{eq:rounding}
  \lceil \lnew_j \rceil \le x_j \le \lfloor \unew_j \rfloor
  \end{equation}
\end{enumerate}
\end{observation}

If the first two steps are not applicable, the algorithm computes the new bounds $\lnew$ and $\unew$ in Steps \ref{dp:stepiii} and \ref{dp:stepiv}. For a given variable $j$, if $\lnew_j > \ell_j$, then the bound is updated with the new value. Similarly, $u_j$ is updated if $\unew_j < u_j$.

An actual implementation may skip Steps~\ref{dp:stepi} and~\ref{dp:stepii} without changing the result.
This is because for redundant constraints Steps~\ref{dp:stepiii} and \ref{dp:stepiv} correctly detect no bound
tightenings, and for infeasible constraints, Steps~\ref{dp:stepiii} and \ref{dp:stepiv} lead to at least one
variable with an empty domain, i.e., $\lnew_j > \unew_j$.

When propagating a system of the type \eqref{eq:MIP} which consists of several constraints, one simply applies the above steps to each constraint independently. Notice that in such systems, it is possible for two or more constraints to share the same variables (i.e., coefficients $a_j$ are non-zero in several constraints). Therefore, if a bound of a variable is changed in one constraint, this can trigger further bound changes in the constraints which also have this variable. This gives the propagation algorithm its iterative nature, as one has to repeat the propagation process over the constraints as long as at least one bound change is found. A pass over all the constraints is also called a \emph{propagation round}. If no bound changes are found during a given round, then no further progress is possible and the algorithm terminates. At this point, all constraints are guaranteed to be bound consistent \cite{Achterberg2009}.

This algorithm can be interpreted as a fixed-point iteration in the space of variable and activity bounds with a unique fixed point \cite{intbdpropcomplexity}; it converges to this fixed point, however not necessarily in finite time \cite{BelottiCafieriLeeLiberti2010}. Additionally, even when it does converge to the fixed point in finite time, convergence can be very slow in practice \cite{Achterberg2009,intbdpropcomplexity,ncsps_lhomme}. To deal with this, practical implementations of bounds propagation introduce tolerance-based termination criteria which stop the algorithm if the progress becomes too slow, i.e., the relative size of improvements on the bounds falls below a specified threshold. With this modification, the algorithm always terminates in finite time (but not in worst-case polynomial-time), however, it may fail to compute the best bounds possible.

To distinguish the above-described approach from alternative methods to compute consistent bounds (see, e.g., \cite{BelottiCafieriLeeLiberti2010} for a method solving a single LP instead), we will use the following definition:
\begin{definition}[Iterative Bounds Tightening Algorithm] 
\label{def:iterative_dom_prop_algorithm}
      Given variable bounds $\ell$, $u$ of a problem of the form \eqref{eq:MIP}, any algorithm updating these bounds by calculating $\lnew$, $\unew$ via \eqref{eq:bound_candidates_pos}, \eqref{eq:bound_candidates_neg}, and \eqref{eq:rounding} iteratively as described in Observation \ref{obs:dom_prop_steps}, thus traversing a sequence of bounds $(\ell, u)_1, (\ell, u)_2, \ldots$ is called an \emph{iterative bounds tightening algorithm} (IBTA).
\end{definition}
Note that this definition leaves the flexibility for individual algorithmic choices, for example, the timing of when bound changes are applied or the order in which the constraints are processed. If a given algorithm applies the found changes immediately, making them available to subsequent constraints in the same iteration, it might traverse a shorter sequence of bounds to the fixed point than the algorithm which delays updates of bounds until the end of the current iteration (e.g., because it processes constraints in parallel). The ordering of processed constraints can lead to different traversed sequences because a given bound change that depends on other changes being applied first might be missed in a given iteration if the constraint it depends on is not processed first.

\subsection{Motivation}
\label{sec:motivation}

Our motivation for this paper is threefold:

1. \textit{Estimating the premature stalling effect of IBTAs:} In the context of MINLP, Belotti et al. \cite{BelottiCafieriLeeLiberti2010} propose an alternative bounds propagation algorithm that computes the bounds at the fixed point directly by solving a single linear program. This approach circumvents non-finite convergence behavior and shows that the bounds at the fixed point can in theory be computed in polynomial time.

Nevertheless, in practice, the trade-off between the quality of obtained bounds including their effect on the wider branch-and-bound algorithm and the algorithm's runtime makes the iterative bounds propagation with tolerance-based stopping criteria the most effective method in most cases, despite its exponential worst-case runtime. The use of stopping criteria still leaves individual instances or potentially even instance classes susceptible to the following effect stated by Belotti et al. as a motivation for their LP-based approach, which is also the motivation for our paper: \emph{``However, because the improvements are not guaranteed to be monotonically non-increasing, terminating the procedure after one or perhaps several small improvements might in principle overlook the possibility of a larger improvement later on.''} In their paper, no attempt is made to quantify this statement, as likely out-of-scope and non-trivial to answer. 

In this work, we aim to develop a methodology to quantify the overall progress that a given IBTA achieved up to a given point in its execution. Ideally, we would like to have a function $f$, which maps current variable bounds to a scalar value, for example in $[0, 100]$, which measures the achieved progress. The main difficulty in developing such a function comes in the form of unbounded variable domains in the input instances (and potentially during the algorithm's execution). Observing the values of such a function over the execution time of the algorithm could then be used to study the behavior of IBTAs on instances of interest and quantify the effect brought up by Belotti et al., which we call \emph{premature stalling} (see Section \ref{sec:results_stalling} for formal definition). Furthermore, an algorithm-independent $f$ would allow comparing the behavior of different IBTAs with respect to premature stalling.

2. \textit{Performance comparison of different IBTAs in practice:} As already motivated by Definition \ref{def:iterative_dom_prop_algorithm}, different IBTAs might traverse different sequences of bounds from the initial values to the fixed point. Additionally, we stated in Section \ref{sec:dom_prop_background} that in practice, iterative bounds propagation is used exclusively with tolerance-based stopping criteria, meaning that the algorithm is stopped potentially before reaching the fixed point. The following problem then arises: for two such algorithms traversing different sequences of bounds that are stopped before reaching the unique fixed point, how do we judge which one performed better? Perhaps a more natural way to formulate this question is: \emph{in how much time do the two algorithms achieve the same amount of progress?} A function measuring the progress of iterative bounds propagation as already proposed can be used to answer this question.

As a concrete example, we will compare the following two IBTAs: the canonical, state-of-the-art sequential implementation, for example from \cite{Achterberg2009}, and a GPU-parallel algorithm recently proposed in \cite{SofranacGleixnerPokutta2020}. In the preliminary computational study on the MIPLIB 2017 test set \cite{GleixnerEtal2019} presented in \cite{SofranacGleixnerPokutta2020}, the two algorithms are compared for the propagation to the fixed point (no tolerance-based stopping criteria). In this work, we will compare the performance of the two algorithms in a real-world setting, i.e., when terminated before reaching the fixed point.

3. \textit{Designing stopping criteria:} as already stated, the tolerance-based stopping criteria are crucial for effective IBTAs. Notice that because different IBTAs might traverse different sequences of bounds, their average individual improvements on the bounds might be different in size. In fact, the study in \cite{SofranacGleixnerPokutta2020} shows that on average, the size of improvements by the GPU-parallel algorithm is smaller than that of the canonical sequential implementation over the MIPLIB 2017 test set, despite its higher performance in terms of runtime to the fixed point. An important implication of this effect is that given two such algorithms, a given stopping criterion might be effective for one of them, but ineffective for the other. In this context, quantifying the magnitude and distribution of expected improvements of a given algorithm for a given problem class and its likelihood to prematurely stall, would allow one to make more informed decisions when designing \emph{effective} stopping criteria.

Lastly, we believe that gaining insight into the behavior of these algorithms is a motivation in itself that could potentially benefit future and existing methodologies in the context of linear constraint propagation.                   

\section{Finite and infinite domain reductions}
\label{sec:finite_and_infinite_domain_reductions}

Any IBTA starts with arrays of initial lower and upper bounds, $\ell^s \in \Rneginf^n$ and $u^s \in \Rposinf^n$, respectively, and incrementally updates individual bounds towards the uniquely defined fixed-point bounds which we denote by $\ell^l$ and $u^l$. To denote the arrays of bounds at any given time between the start and the fixed point we simply use $\ell$ and $u$ and call them \emph{current bounds}. Obviously, it holds that $\ell^s_j \le \ell_j \le \ell^l_j$ and $u^s_j \ge u_j \ge u^l_j$ for all $j \in \{1,\ldots,n\}$. Observe that both initial and limit bounds may contain infinite values.

\subsection{Reducing Infinite Bounds to Finite Values}
\label{sec:inf_phase}

Variables that start with infinite value in either lower or upper bound, will either remain infinite if no bound change is possible or will become finite values. We start with the following simple observation:
\begin{observation}
\label{prop:no_val_dependence}  
Given a constraint of the form \eqref{eq:lincons} and a given variable $j \in \{ 1,\ldots,n \}$ with a bound $\ell_j = -\infty$ (or $u_j = \infty$), the possibility of tightening this bound to some finite value depends on the signs of coefficients $a_j, j \in \{1,\ldots,n\}$, the finiteness of variable bounds $\ell_j, j \in \{1,\ldots,n\} \setminus j$ and $u_j, j \in \{1,\ldots,n\} \setminus j$, and the finiteness of $\underline{\beta}$ and $\overline{\beta}$, but not on the values that these variables take, if they are finite.
\end{observation}
\begin{proof}
To see the dependence on the sign of coefficients $a$, let the lower bound of a given variable $j$ be $\ell_j = -\infty$  and let $\underline{\beta}$ and $\overline{\alpha}_j$ be finite, $\overline{\beta} = \infty$ and $\underline{\alpha}_j = -\infty$. Then, by \eqref{eq:bound_candidates_pos} and \eqref{eq:bound_candidates_neg}, $a_j > 0$ implies $\ell^{\text{new}}_j \in \mathbb{R} > -\infty$ and the bound is updated. Else, if $a_j < 0$, then $\ell^{\text{new}}_j = -\infty$ and no bound change is possible.

The dependence on the finiteness of $\underline{\beta}$ and $\overline{\beta}$ is trivial, while the coefficients $a$ are finite by problem definition. To see the dependence on the finiteness of variable bounds, consider the activities $\underline{\alpha}_j$ and $\overline{\alpha}_j$ of a variable $j$ with $\ell_j=-\infty$, $u_j=\infty$, and $a_i>0$ for all $i \in \{1,\ldots,n\}$. If there exists $k$ such that $u_k = \infty, k \in \{1,\ldots,n\} \setminus j$ then $\overline{\alpha} = \infty$ and consequently $\ell_j$ cannot be tightened. Otherwise, if $u_k \in \mathbb{R}$ for all $k \in \{1,\ldots,n\} \setminus j$, then $\overline{\alpha} \in \mathbb{R}$ and a bound tightening is possible.

The specific finite values that the variables in \eqref{eq:bound_candidates_pos} and \eqref{eq:bound_candidates_neg} take have no effect on the possibility to reduce an infinite bound to a finite value because arithmetic operations between finite values again produce a finite value (also $a_j \ne 0$ by definition) and $-\infty < k < \infty$ for all $k \in \mathbb{R}$. Variables which are restricted to integer values also do not affect this process, as the operations $\lceil \ell_j \rceil$ and $\lfloor u_j \rfloor$ give $\ell_j, u_j \in \mathbb{Z}$ and $\mathbb{Z} \subset \mathbb{R}$. The same argument from above then applies.        
\end{proof}

Notice that the same effect of finite bound changes triggering new bound changes in the subsequent propagation rounds is also true for infinite domain reductions, hence this process might also require more than one iteration. Furthermore, these iterations have the following property:

\begin{corollary}
\label{prop:n_rounds_of_inf_phase}
Let $k \in \mathbb{N}$ be the number of iterations a given IBTA takes to reach the fixed point. Then there is a number $c \leq k$, $c \in \mathbb{N}_0$ such that the first $c$ propagation rounds have at least 1 reduction of an infinite to a finite bound, and none thereafter. By pigeonhole principle, $c$ is at most the number of initially infinite bounds.
\end{corollary}

\begin{proof}
The coefficients $a$ and the left- and right-hand sides $\underline{\beta}$ and $\overline{\beta}$ are constants that do not change during the course of the algorithm. By Observation \ref{prop:no_val_dependence} the only thing left influencing the infinity reductions is the finiteness of variable bounds. If no infinite to finite reductions are made at any given round, then none can be made thereafter. Finite to infinite reductions are not possible as the algorithm only accepts improving bounds.
\end{proof}

In conclusion, the process of reducing infinite bounds to some finite values is independent and fundamentally different from the incremental improvements of finite values thereafter, which is driven by the values of the variables in \eqref{eq:bound_candidates_neg} and \eqref{eq:bound_candidates_pos}. Accordingly, we will measure the ability of an algorithm to reduce the infinite bounds to some finite values separately from its ability to make improvements on the finite values of the bounds thereafter.

\subsection{Finite Domain Reductions}
\label{sec:finite_phase}

Our main approach in measuring the progress of finite domain reductions (see Section \ref{sec:measuring_finite_phase}) relies on the observation that the starting as well as the fixed point of propagation is uniquely defined for a given MIP problem and hence independent from the algorithm used. The measuring function then answers the following question: for given bounds $\ell$ and $u$ at some time during the propagation process, how far have we gotten from the starting point $\ell^s$ and $u^s$, relative to the endpoint $\ell^l$ and $u^l$. When the bounds of a given variable did not change during the propagation process, or they are finite at both the start and the end, there is no difficulty in calculating such a measure. However, when a given variable bound started as an infinite value but was tightened to some finite value by the end of propagation, special care is needed to handle this case, which we address in this section.

In Section \ref{sec:background_and_motivation}, we discussed how a sequential and a parallel propagation algorithm might traverse different sequences of bounds during their executions. Let us consider the first round of two such algorithms, and see what might happen to the bounds which start as infinite but are tightened during the course of the algorithm. When the sequential algorithm finds a bound change, it is immediately made available to the subsequent constraints in the same round. Consequently, if an infinite domain reduction happens in the subsequent constraints, it may produce a stronger finite value compared to the parallel algorithm which used the older (weaker) bound information. This serves to show that the first finite values that such bounds take may not be the same in different IBTAs. Hence they cannot be used safely to compare finite domain reductions across different implementations. In what follows, we construct a procedure to compute algorithm-independent reference values for each bound.

\begin{definition}[weakest variable bounds]
\label{def:weakest_bounds}
Given an optimization problem of the form \eqref{eq:MIP} with starting variable bounds $\ell^s$ and~$u^s$, we call $\overline{\ell}_j$
\emph{weakest lower bound} of variable $j$ if
\begin{itemize}
\item $\lw_j=-\infty$ and no IBTA can produce a finite lower bound $\ell_j\in\R$, or
\item $\lw_j\in\R$ and no IBTA can produce a finite lower bound $\ell_j\in\R$ with $\ell_j < \lw_j$.
\end{itemize}
We call $\overline{u}_j$ \emph{weakest upper bound} of variable $j$ if
\begin{itemize}
\item $\uw_j=\infty$ and no IBTA can produce a finite upper bound $u_j\in\R$, or
\item $\uw_j\in\R$ and no IBTA can produce a finite upper bound $u_j\in\R$ with $u_j > \uw_j$.
\end{itemize}
\end{definition}

When both the starting and the limit bounds are finite we have $\lw_j = \ell_j^s$ resp.~$\uw_j = u_j^s$ (because all IBTAs only accept improving bounds) and when they are both infinite we have $\ell^s_j = \ell_j^l = \lw_j = -\infty$ resp.~$u_j^s = u_j^l = \uw_j = \infty$. Notice that the cases of $\ell_j^s \in \R$ with $\ell_j^l = -\infty$ and $u_j^s \in \R$ with $u_j^l = \infty$ are not possible as $\ell^s_j \le \ell^l_j$ and $u^s_j \ge u^l_j$. The main challenge in computing $\lw$ and $\uw$ is due to the remaining case of $\ell_j^s = -\infty, \ell^l_j \in \R$ resp.~$u_j^s = \infty, u_j^l \in \R$. In what follows, we will extend the notation introduced in Section \ref{sec:dom_prop_background} with $\Bijsl$ and $\Bijsu$ denoting $\Bsl$ and $\Bsu$ applied to constraint $i$ and variable $j$, respectively. The procedure presented in Algorithm \ref{alg:weakest_bounds} computes $\lw$ and $\uw$.

\begin{algorithm}[htb]
\caption{The Weakest Bounds Algorithm}
\label{alg:weakest_bounds}
\begin{algorithmic}[1]
\REQUIRE System of $m$ linear constraints $\underline{\beta} \le \sum_{i=1}^n a_ix_i \le \overline{\beta}$, $\ell \le x \le u$ 
\ENSURE Weakest variable bounds $\lw$ and $\uw$

\STATE mark all constraints \label{markallcons}
\STATE \code{bound\_change\_found} $\gets$ \code{true}
\STATE $\lw = \ell, \uw = u$

\WHILE{\code{bound\_change\_found}}

\STATE \code{bound\_change\_found} $\gets$ \code{false}

\FOR{\textbf{each} constraint $i$} \label{consloop}
\IF{$i$ marked} \label{ifconsmarked}
\STATE unmark $i$ \label{unmarkcons}
\FOR{\textbf{each} variable $j$ such that $a_{ij} \ne 0$}

\IF{$a_{ij} > 0$}
\STATE $\lnew_j = \Bijsl(\lw, \uw)$
\STATE $\unew_j = \Bijsu(\lw, \uw)$
\ELSE
\STATE $\lnew_j = \Bijsu(\lw, \uw)$
\STATE $\unew_j = \Bijsl(\lw, \uw)$
\ENDIF

\IF{$x_j \in \mathbb{Z}$}
\STATE $\lnew = \lceil \lnew_j \rceil, \unew_j = \lfloor \unew_j \rfloor$
\ENDIF

\IF{$\ell_j = -\infty$ and $\lnew_j \in \R$ and $(\lw_j = -\infty$ or $(\lw_j \in \R$ and $\lnew_j < \lw_j))$ } \label{alg:weakest_l_check}
\STATE $\lw_j \gets \lnew_j$
\STATE \code{bound\_change\_found} $\gets$ \code{true}
\ENDIF

\IF{$u_j = \infty$ and $\unew_j \in \R$ and $(\uw_j = \infty$ or $(\uw_j \in \R$ and $\unew_j > \uw_j))$} \label{alg:weakest_u_check}
\STATE $\uw_j \gets \unew_j$
\STATE \code{bound\_change\_found} $\gets$ \code{true}
\ENDIF

\IF{\code{bound\_change\_found}} \label{ifbdchgfound}
\STATE mark all constraints $k$ such that $a_{kj} \ne 0$ \label{markconswithvar}
\ENDIF

\ENDFOR
\ENDIF
\ENDFOR
\ENDWHILE

\RETURN $\lw$, $\uw$
\end{algorithmic}
\end{algorithm}

The procedure starts by setting $\lw = \ell^s$ and $\uw = u^s$ and will proceed to iteratively update these bounds until they are all weakest bounds. Up to Lines \ref{alg:weakest_l_check} and \ref{alg:weakest_u_check}, the procedure is very similar to the usual bounds propagation: it evaluates \eqref{eq:bound_candidates_pos}, \eqref{eq:bound_candidates_neg}, and \eqref{eq:rounding} on the latest available bounds for all constraints and variables. As the bounds which start as finite values are already weakest by definition, the first part of the checks in Lines \ref{alg:weakest_l_check} and \ref{alg:weakest_u_check} makes sure that these variables are not considered. For bounds that are infinite at the start, the algorithm checks if the new candidate is finite. If so, the new candidate becomes the weakest bound incumbent if the current weakest bound is infinite, or the new candidate is weaker than the current one. This process then repeats in iterations until no further weakenings are possible. Notice that the \emph{constraint marking mechanism}, implemented in Lines \ref{markallcons}, \ref{ifconsmarked}, \ref{unmarkcons}, \ref{ifbdchgfound}, and \ref{markconswithvar} is not necessary for the correctness of the weakest bounds procedure, but as it can substantially speed up the execution of the algorithm, we include it in the pseudocode.

\section{An Algorithm-Independent Measure of Progress}  
 \label{sec:measuring_function}

 As pointed out in Section \ref{sec:inf_phase}, we will measure the ability of an IBTA to reduce infinite bounds to some finite values separately from the improvements of finite bounds. Section \ref{sec:measuring_infinite_phase} presents the functions measuring infinite domain reductions, while Section \ref{sec:measuring_finite_phase} presents the functions measuring the progress in finite domain reductions.

 As before, we denote the starting bounds of a variable $j$ as $\ell^s_j$ and $u^s_j$, the weakest bounds as $\lw_j$ and $\uw_j$, the limit bounds as $\ell^l_j$ and $u^l_j$, and the bounds at a given point in time during the propagation as $\ell_j$ and $u_j$. Recall that the following relations hold: $\ell^s_j \le \ell_j \le \ell^l_j$ and $u^s_j \ge u_j \ge u^l_j$ for all $j \in \{1,\ldots,n\}$. Additionally, if $\ell_j \in \R$, then $\ell^l_j \in \R$ and $\ell^s_j \le \lw_j \le \ell_j \le \ell^l_j$. Likewise, if $u_j \in \R$ then $u^l_j \in \R$ and $u^s_j \ge \uw_j \ge u_j \ge u^l_j$. Lastly, if $\ell_j^s \in \R$ then $\lw_j = \ell^s_j$ and if $u^s_j \in \R$ then $\uw_j = u^s_j$. Figure \ref{fig:bounds_on_real_line} illustrates example starting, current, and limit bounds of a given variable on the real line. 
\begin{figure}
  \centerline{\includegraphics[width=0.75\textwidth]{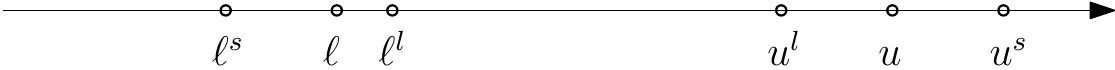}}
\caption{Schematic representation of the starting (index $s$), current (no index) and limit bounds (index $l$) for a given variable on the real line. In this example, $\ell^s = \lw \in \R$ and $u^s=\uw \in \R$.}
\label{fig:bounds_on_real_line}
\end{figure}
\subsection{Measuring Progress in Infinite Domain Reductions}
\label{sec:measuring_infinite_phase}

As bounds propagation has a unique fixed point to which it converges, we know the state of the algorithm at both the beginning and the end (a given bound is either finite or infinite). Denote by $\ntotal \in \mathbb{N}$ the total number of bounds that change from an infinite to some finite value between the starting and the limit bounds of the problem, and by $\ncurrent \in \mathbb{N} \le \ntotal$ the number of infinite bounds reduced to finite values by a given IBTA at a given point during its execution:
\begin{subequations}
\begin{align}
& \ntotal = |\{ j=1,\ldots,n : \ell^s_j = -\infty, \ell^l_j \in \R \}| + |\{ j=1,\ldots,n : u^s_j = \infty, u^l_j \in \R \}|,
\intertext{and, }
& \ncurrent = |\{ j=1,\ldots,n : \ell^s_j = -\infty, \ell_j \in \R \}| + |\{ j=1,\ldots,n : u^s_j = \infty, u_j \in \R \}|.
\end{align}
\end{subequations}
\noindent
Then, the progress in infinite domain reductions of the IBTA at that point is calculated as:
\begin{equation}
\Pinfinite = \frac{\ncurrent}{\ntotal}, \ntotal \ne 0.
\end{equation}
Observe that the total number of infinite domain reductions $\ntotal$ is algorithm-independent and can be precomputed from the starting and the limit bounds. Because IBTAs never relax bounds, $\Pinfinite$ is trivially non-decreasing.
  
\subsection{Measuring Progress in Finite Domain Reductions}
\label{sec:measuring_finite_phase}

The concept of the weakest variable bounds developed in Section \ref{sec:finite_phase} gives us a natural starting point for finite domain reductions. As bounds propagation converges towards its unique fixed point, the endpoint is also well defined. Notice that the bounds which are infinite at the endpoint, also had to be infinite at the starting point, meaning that no change was made on this bound. The rest of the bounds are either infinite at the beginning, in which case we can compute the weakest bound by Algorithm \ref{alg:weakest_bounds}, or the bound is finite at both the start and the end.

Our main approach is to measure the relative progress of each individual bound from its weakest value towards the limit value. Given a variable $j \in \{1,\ldots,n\}$ we will denote by $\Prog_{\ell_j} \in \R$ and $\Prog_{u_j} \in \R$ the scores which measure the amount of progress made on its lower and upper bounds $\ell_j$ and $u_j$, respectively, at a given point in time. Afterward, we will combine the scores of all the variable bounds into the global progress in the form of a single scalar value $\Pfinite \in \R$, which measures the global progress in finite domain reductions at a given point in time.

For variable~$j$, $\Prog_{\ell_j}$ and $\Prog_{u_j}$ are computed as
\begin{subequations}
\begin{align}
\label{P_l_j}
& \Prog_{\ell_j} =
  \begin{cases}
      \frac{\ell_j - \lw_j}{\ell^l_j - \lw_j} & \text{ if } \ell_j > \lw_j \text{ and } \lw_j \ne \ell^l_j \\
      0 & \text{ otherwise}
    \end{cases},
\intertext{and}
\label{P_u_j}
 & \Prog_{u_j} =
   \begin{cases}
      \frac{\uw_j - u_j}{\uw_j - u^l_j} & \text{ if } u_j < \uw_j \text{ and } \uw_j \ne u^l_j \\
      0 & \text{ otherwise}
    \end{cases}.
\end{align}
\end{subequations}
Given the vectors of scores for individual bounds $\Prog_{\ell} \in \R^n$ and $\Prog_u \in \R^n$, we calculate $\Pfinite$ as
\begin{equation}
\Pfinite = \lVert \Prog_\ell \lVert_1 + \lVert \Prog_u \lVert_1 = \sum_j (\Prog_{\ell_j} + \Prog_{u_j}),
\end{equation} 
where $\lVert \cdot \lVert_1$ denotes the $\ell_1$ norm. It holds that $\Prog_{\ell_j}, \Prog_{u_j} \in [0, 1]$ and
\begin{equation}
  \Pfinite \leq |\{ j=1,\ldots,n : \lw_j \ne \ell^l_j \}| + |\{ j=1,\ldots,n : \uw_j \ne u^l_j \}|.
\end{equation}
This maximum score is algorithm-independent and can be precomputed for each instance. This makes it possible to normalize the maximum score to, e.g., $100\%$. Again, because IBTAs never relax bounds, this progress function is trivially non-decreasing.

\subsection{Implementation Details}
\label{sec:implementation}

To precompute $\lw$ and $\uw$, we implemented Algorithm \ref{alg:weakest_bounds}. To obtain $\ell^l$ and $u^l$, any correct bounds propagation algorithm can be run on the original problem, assuming that it propagates the problem to the fixed point (no tolerance-based stopping criteria).

Computing the progress measure is expensive relative to the amount of work that bounds propagation normally performs. Hence, it can considerably slow down the execution and incur unrealistic runtime measurements. To avoid this effect, we proceed as follows in our implementation. First, we run the bounds propagation algorithm together with progress measure computation and record the scores after each round. Then, we run the same bounds propagation algorithm but without the progress measure calculation and record the time elapsed to the end of each round. This gives us progress scores and times for each round, but also the time it took to reach the scores at the end of each round.

\section{Applications of the Progress Measure}
\label{sec:applications}

In this section, we apply the progress measure in order to answer two questions of practical relevance. In Section \ref{sec:experimental_setup}, we first describe the experimental setup that will form the base for subsequent evaluations. In Section \ref{sec:results_stalling} we show that MIP instances in practice rarely cause IBTAs to stall prematurely, i.e., have very slow progress followed by larger improvements thereafter, a concern brought up in \cite{BelottiCafieriLeeLiberti2010} (see Section \ref{sec:motivation}). In Section \ref{sec:results_gpu_dom_prop}, we show that the newly-developed GPU-based propagation algorithm from \cite{SofranacGleixnerPokutta2020} is even more competitive in a practical setting than reported in the original paper.

\subsection{Experimental Setup}
\label{sec:experimental_setup}

We will refer to two linear constraint propagation algorithms: 
\begin{enumerate}
\item \textbf{\gpu} is the GPU-based algorithm from \cite{SofranacGleixnerPokutta2020}, and
\item \textbf{\seq} is the canonical sequential propagation as described in e.g. \cite{Achterberg2009}. Our implementation closely follows the implementation in the academic solver SCIP \cite{GamrathEtal2020OO}.
\end{enumerate}
We use the MIPLIB~2017 test set, which is currently the most adopted and widely used testbed of MIP instances \cite{GleixnerEtal2019}. This test set contains \num{1065} instances, however, the open-source MIP file reader we used had problems with reading \num{133} instances, leaving the test set at \num{932} instances. On \num{72} instances \gpu and \seq failed to obtain the same fixed point (due to e.g., numerical difficulties and other problems), and we remove these instances from the test set as well. Additionally, we impose an iteration limit of \num{100} for both propagation algorithms, with \num{2} instances hitting this limit.

During MIP solving, the case where no bound changes are found during propagation is valid and common. However, this is of no interest to us here, as we could make no measurements of progress. There are \num{310} such instances in the test set. Furthermore, \num{8} instances with challenging numerical properties showed inconsistent behavior with our implementations, and we remove these instances from the test set as well. Finally, the test set used for the evaluations is left with \num{540} MIP instances.

In terms of hardware, we execute the \gpu algorithm on a NVIDIA Tesla V100 PCIe 32GB GPU, and the \seq algorithm on a 24-core Intel Xeon Gold 6246 @ 3.30GHz with 384 GB RAM CPU. All executions are performed with double-precision arithmetic.

As we use this test set to measure the progress of propagation algorithms, they were run until the fixed point is reached with the progress recorded as described in Section \ref{sec:implementation}. In this setting, IBTAs terminate after no bound changes are found at a given propagation round. What this means is that the last two rounds will both have the same maximum score (no bound changes in the last round). Because this feature reflects the design of the algorithms, in the results we assume that the maximum score is reached after the last round, and not after the second-to-last round. This is equivalent to removing the second-to-last round. On the other hand, when the (finite or infinite) score does not change its value between two rounds which are not the last and the second-to-last one, we assume that the score is reached at the first time when it is recorded.

Due to implementation reasons, we will sample progress after each propagation round of an algorithm, rather than after every single bound change. Then, we use linear interpolation to build the progress functions $\Pfinite$ and $\Pinfinite$ and thus obtain an approximation of the true progress function.

\subsection{Analyzing Premature Stalling in Linear Constraint Propagation}
\label{sec:results_stalling}

First, we have to quantitatively define the premature stalling effect. The danger it poses is that the stopping criteria might terminate the algorithm after an iteration with slow progress, and potentially miss on substantial improvements later on. While infinite domain reductions are usually easy to find by bounds propagation algorithms, they are nevertheless considered significant and the algorithm is usually not stopped after an iteration that contains these tightenings \cite{Achterberg2009}. Accordingly, we will reflect this in our premature stalling effect definition.

We slightly adapt the notation introduced in Section \ref{sec:measuring_finite_phase} and define the progress in finite domain reductions as a function of time denoted by $\Prog: [0, 100] \rightarrow [0, 100]$. Observe that the input (time) and output (progress) of this function are normalized to values between \num{0} and \num{100}. In this notation we assume that $\Prog$ is continuous and twice differentiable, however, in practice, the progress is sampled only after each propagation round and $\Prog$ built by linear interpolation. In our implementation, we approximate the derivatives of $\Prog$ by second-order accurate central differences in the interior points and either first or second-order accurate one-sided (forward or backward) differences at the boundaries \cite{quarteroni2007numerical,Fornberg1988GenerationOF}. Additionally, given a propagation round $r$, $t(r)$ denotes the normalized time at the end of propagation round $r$. All derivates are w.r.t. time: $\Prog' = \frac{d}{dt}\Prog$. We denote by $k \in \mathbb{N}$ the number of iterations the propagation algorithm takes to reach the fixed point and by $\ell^r, u^r \in \R^n$ the arrays of lower and upper bounds at iteration $r$, respectively. Then, the premature stalling effect is defined as follows.
\begin{definition}
\label{def:stalling}
Let \Prog be a progress function of finite domain reductions for the propagation of a given MIP instance. Then, the propagation algorithm is said to \emph{prematurely stall} with coefficients $p,q \in \R_{\infty \ge 0}$ at round $r \in \{2,\ldots,k\}$ if the following conditions are true:
\begin{enumerate}
  \item there does not exist $j \in \{1,\ldots,n\}$ such that $\ell^{r-1}_j = -\infty$ and $\ell^{r}_j \in \R$,
  \item there does not exist $j \in \{1,\ldots,n\}$ such that $u^{r-1}_j = \infty$ and $u^{r}_j \in \R$,
  \item $\Prog'(t(r)) < p $, and
  \item there exists $x \in [t(r), 100]$ such that $\Prog''(x) > q$.
  \end{enumerate}
\end{definition}

The first two conditions simply state that there were no infinite domain reductions in round $r$. To understand the third condition, let $p = 0.1$ at $r$. This would mean that the algorithm is progressing at a rate of \num{1} percent of progress in \num{10} percent of the time at $r$ (recall the normalized domains of \Prog). Taking another derivative and looking at the remainder of the time interval reveals if this rate will increase (is greater than \num{0}), meaning that there are bigger improvements to follow than the improvements the algorithm is currently making. The parameter $q \ge 0$ allows quantification of increase in size of these improvements. Also, recall from Section \ref{sec:measuring_finite_phase} that $\Prog$ is non-decreasing and hence $\Prog'(t) \ge 0$ for all $t \in [0, 100]$. With this, we can now detect instances where slow progress is followed by a significant increase in improvements.

Table \ref{tab:table} reports the number of premature stalls in the test set for several different combinations of parameters $p$ and $q$. Notice that the \num{310} instances for which no bound changes are found cannot stall by definition. Additionally, \num{57} instances in the test set only recorded infinite domain reductions, and these instances also cannot prematurely stall by definition. The results of testing the remaining \num{432} instances which do record at least one finite domain reduction for premature stalling are shown in Table \ref{tab:table}.

\begin{table}[ht!]
  \begin{center}
    \caption{Number of premature stalls in the test for different values of parameters $p$ and $q$.}
    \label{tab:table}
    \begin{tabular}{cccc}
      \toprule
                 &            & \multicolumn{2}{c}{\textbf{\# stalls}}\\
      \textbf{p} & \textbf{q} & \seq & \gpu\\
      \midrule
      $\infty$ & 0.0 & 48 & 44\\
      0.1      & 0.0 & 14 & 18\\
      0.1      & 0.2 & 1  & 0\\
      0.1      & 0.5 & 0  & 0\\
      0.5      & 0.5 & 1  & 0\\
      0.5      & 2.0 & 0  & 0\\
      \bottomrule
    \end{tabular}
  \end{center}
\end{table}
Let us first look into the results for \seq. From the first row of the table, we can see that only \num{48} instances experience any kind of increase in the second derivative during the execution, i.e., the improvements get smaller in time for all but \num{48} instances in the test set (equivalently, $\Prog$ is concave for all but \num{48} instances). From the second row, we can see that among these \num{48} instances that experience any kind of second derivative increase, \num{14} experience slow progress of $p=0.1$ at least once during their execution. Among these, only \num{1} instance experiences an increase in second derivative  of more than $q=0.2$ following the slow progress of $p=0.1$. If we further restrict the increase in the second derivative to $q=0.5$, then no instances are shown to stall prematurely. In the last row we see that even if the slow progress is relaxed to $p=0.5$, there are no instances that record a more significant increase in the second derivative of $2.0$.

Additionally, even though \gpu performed similarly to \seq with respect to stalling, we can still observe that it is on average less susceptible to premature stalling than \seq, as it recorded a smaller or equal amount of instances with premature stalling for all but one parameter combinations. 

We conclude that in practice, the premature stalling effect seems to occur only rarely and on individual instances. This shows that termination criteria based on local progress are reasonable.

\subsection{Analyzing GPU-parallel Bounds Propagation in Practice}
\label{sec:results_gpu_dom_prop}

As pointed out in Section \ref{sec:motivation}, \gpu traverses a potentially different sequence of bounds from the start to the fixed point than \seq. Because of this, computational experiments in \cite{SofranacGleixnerPokutta2020} report the speedup of \gpu over \seq for propagation runs to the fixed point. As bounds propagation is stopped early in practice, we will now use the progress measure to compare the two algorithms when stopped at different points in the execution. For each instance in the set, given a progress value $x \in [0, 100]$, the speedup of \gpu over \seq is computed by ${t_x^{\seq}}/{t^{\gpu}_x}$, where $t_x$ is the wall-clock time the algorithm takes to reach progress value $x$.\footnote{For $x=100$, we get the identical speedup at the fixed point evaluation as done in \cite{SofranacGleixnerPokutta2020}.} Then, the geometric mean of speedups over all the instances in the test set is reported. The results are shown on Figure \ref{fig:gpu_results}. When a given instance only has bound changes in the infinite phase, it is excluded from the finite phase comparisons (\num{57} instances). Likewise, instances with only finite progress are removed from the infinite phase (\num{164} instances). 

\begin{figure*}[htbp]
  \centerline{\includegraphics[width=0.7\textwidth]{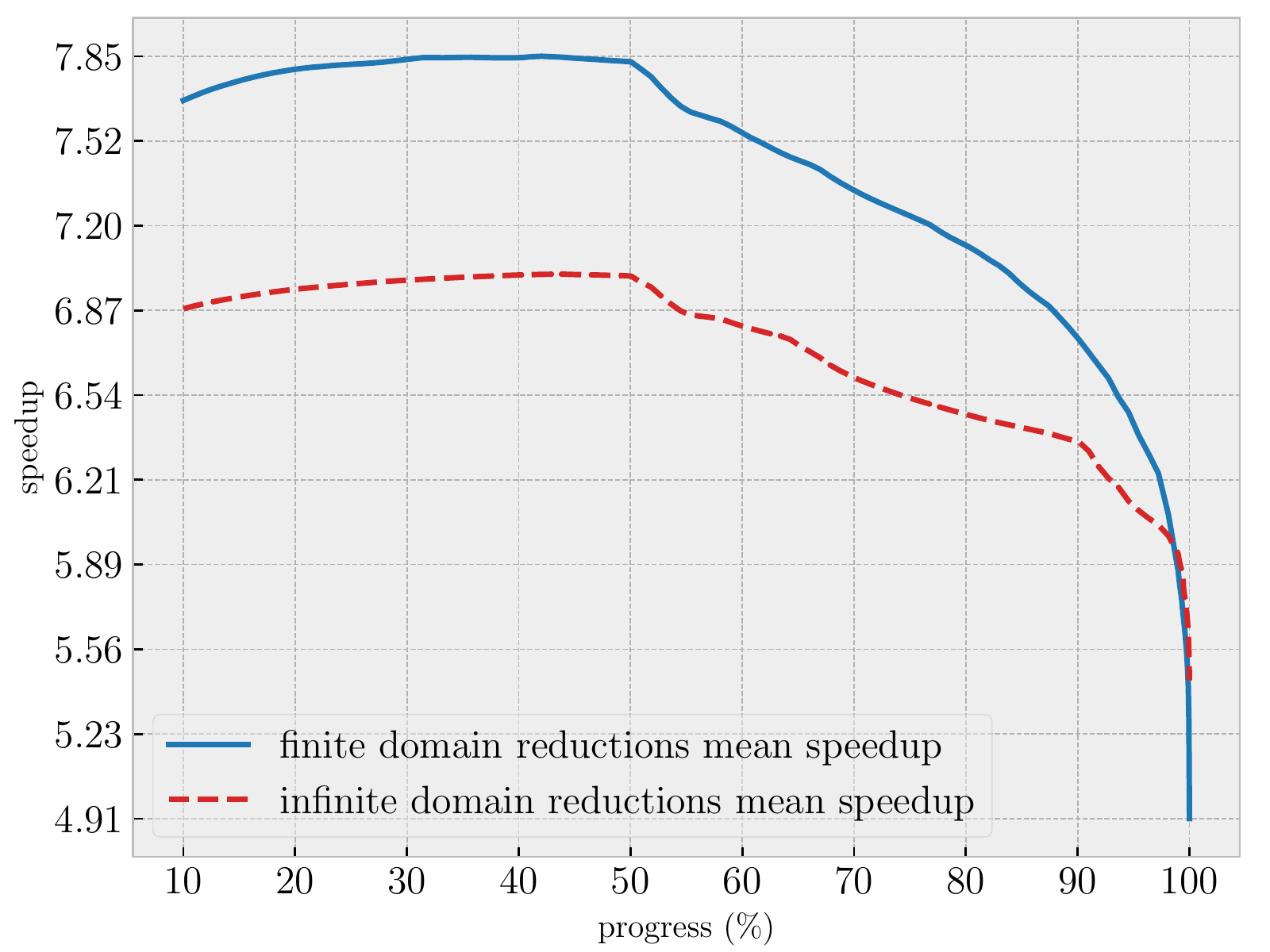}}
\caption{Speedup of the finite and the infinite domain reductions of \gpu over \seq for different percentages of progress made.}
\label{fig:gpu_results}
\end{figure*}

As we can see, for the propagation to the fixed point (\num{100} percent progress), \gpu is about \num{4.9} times faster than \seq in finite domain reductions. For the infinite domain reductions, \gpu is a factor of about \num{5.4} times faster than \seq. Next, we can see that the speedup is minimal at the fixed point, i.e., for any progress value between \num{10} and \num{100}, \gpu increases its speedup over \seq compared to the fixed-point speedup. The maximum speedups of around \num{7.8} for the finite domain reductions and about \num{7.0} for infinite domain reductions are achieved at the progress of roughly \num{50} percent. Additionally, notice that in the last few percent of progress there is a steep drop in speedup. This means that even for very weak stopping criteria which would stop the algorithms at the same point just before the limit is reached, \gpu would significantly increase its speedup over \seq. We conclude that \gpu is even more competitive against \seq in conjunction with stopping criteria than for the case of propagation to the fixed point.

\section{Outlook}
\label{sec:outlook}

In this work, we proposed a method to measure progress achieved by a given algorithm in the propagation of linear constraints with continuous and/or discrete variables. We showed how such a measure can be used to answer questions of practical relevance in the field of Mixed-Integer Programming.

One question that remains open is to what extent the finite reference bounds
produced by the weakest bounds procedure used here are actually realized by at
least one iterative bounds tightening algorithm.  The current procedure only
guarantees that they are finite if iterative bounds propagation can produce a
finite bound, and that no iterative bounds propagation algorithm can produce a
weaker bound.  A deeper analysis could yield a refined method to produce weakest
bounds that are tightest in the sense that they are actually achieved by at
least one iterative bounds propagation algorithm. This is part of future
research and could provide a stronger version of the framework.

Though our development was described for linear constraints, there are no conceptual barriers that prevent the notion of weakest bounds to be extended to more general classes of constraints. We demonstrated how the key issue of unbounded variable domains can be solved in order to obtain an algorithm-independent measure of progress. In this sense, our method is also relevant for constraint systems on (partially) unbounded domains, where normalization can be nontrivial. An important example is the class of factorable programs from the field of Global Optimization and Mixed-Integer Nonlinear Programming.

\bibliographystyle{abbrv}
\bibliography{measure}

\begin{thebibliography}{10}

\bibitem{branch_and_infer}
Branch and infer: A unifying framework for integer and finite domain constraint
  programming.
\newblock {\em INFORMS J. on Computing}, 10(3):287–300, Aug. 1998.

\bibitem{Achterberg2009}
T.~Achterberg.
\newblock {\em Constraint Integer Programming}.
\newblock PhD thesis, TU Berlin, 2009.

\bibitem{Achterberg_2009_2}
T.~Achterberg.
\newblock {SCIP}: solving constraint integer programs.
\newblock {\em Mathematical Programming Computation}, 1(1):1--41, jan 2009.

\bibitem{doi:10.1287/ijoc.2018.0857}
T.~Achterberg, R.~E. Bixby, Z.~Gu, E.~Rothberg, and D.~Weninger.
\newblock Presolve reductions in mixed integer programming.
\newblock {\em INFORMS Journal on Computing}, 32(2):473--506, 2020.

\bibitem{AchterbergWunderling2013}
T.~Achterberg and R.~Wunderling.
\newblock {\em Mixed Integer Programming: Analyzing 12 Years of Progress},
  pages 449--481.
\newblock Springer Berlin Heidelberg, Berlin, Heidelberg, 2013.

\bibitem{scil}
E.~Althaus, A.~Bockmayr, M.~Elf, M.~J{\"u}nger, T.~Kasper, and K.~Mehlhorn.
\newblock Scil --- symbolic constraints in integer linear programming.
\newblock In R.~M{\"o}hring and R.~Raman, editors, {\em Algorithms --- ESA
  2002}, pages 75--87, Berlin, Heidelberg, 2002. Springer Berlin Heidelberg.

\bibitem{simpl}
I.~Aron, J.~N. Hooker, and T.~H. Yunes.
\newblock Simpl: A system for integrating optimization techniques.
\newblock In J.-C. R{\'e}gin and M.~Rueher, editors, {\em Integration of AI and
  OR Techniques in Constraint Programming for Combinatorial Optimization
  Problems}, pages 21--36, Berlin, Heidelberg, 2004. Springer Berlin
  Heidelberg.

\bibitem{BelottiCafieriLeeLiberti2010}
P.~Belotti, S.~Cafieri, J.~Lee, and L.~Liberti.
\newblock Feasibility-based bounds tightening via fixed points.
\newblock In W.~Wu and O.~Daescu, editors, {\em Combinatorial Optimization and
  Applications, Proc. of COCOA 2010}, pages 65--76, Berlin, Heidelberg, 2010.
  Springer Berlin Heidelberg.

\bibitem{intbdpropcomplexity}
L.~Bordeaux, G.~Katsirelos, N.~Narodytska, and M.~Y. Vardi.
\newblock The complexity of integer bound propagation.
\newblock {\em J. Artif. Int. Res.}, 40(1):657–676, Jan. 2011.

\bibitem{choi}
C.~W. Choi, W.~Harvey, J.~H.~M. Lee, and P.~J. Stuckey.
\newblock Finite domain bounds consistency revisited.
\newblock In A.~Sattar and B.-h. Kang, editors, {\em AI 2006: Advances in
  Artificial Intelligence}, pages 49--58, Berlin, Heidelberg, 2006. Springer
  Berlin Heidelberg.

\bibitem{Fornberg1988GenerationOF}
B.~Fornberg.
\newblock Generation of finite difference formulas on arbitrarily spaced grids.
\newblock {\em Mathematics of Computation}, 51:699--706, 1988.

\bibitem{GamrathEtal2020OO}
G.~Gamrath, D.~Anderson, K.~Bestuzheva, W.-K. Chen, L.~Eifler, M.~Gasse,
  P.~Gemander, A.~Gleixner, L.~Gottwald, K.~Halbig, G.~Hendel, C.~Hojny,
  T.~Koch, P.~Le~Bodic, S.~J. Maher, F.~Matter, M.~Miltenberger, E.~M{\"u}hmer,
  B.~M{\"u}ller, M.~E. Pfetsch, F.~Schl{\"o}sser, F.~Serrano, Y.~Shinano,
  C.~Tawfik, S.~Vigerske, F.~Wegscheider, D.~Weninger, and J.~Witzig.
\newblock {The SCIP Optimization Suite 7.0}.
\newblock Technical report, Optimization Online, March 2020.

\bibitem{GleixnerEtal2019}
A.~Gleixner, G.~Hendel, G.~Gamrath, T.~Achterberg, M.~Bastubbe, T.~Berthold,
  P.~M. Christophel, K.~Jarck, T.~Koch, J.~Linderoth, M.~L\"ubbecke, H.~D.
  Mittelmann, D.~Ozyurt, T.~K. Ralphs, D.~Salvagnin, and Y.~Shinano.
\newblock {MIPLIB 2017: Data-Driven Compilation of the 6th Mixed-Integer
  Programming Library}.
\newblock Technical report, Optimization Online, 2019.

\bibitem{Harvey2003}
W.~Harvey and P.~J. Stuckey.
\newblock Improving linear constraint propagation by changing constraint
  representation.
\newblock {\em Constraints}, 8(2):173--207, 2003.

\bibitem{KochMartinPfetsch2013}
T.~Koch, A.~Martin, and M.~E. Pfetsch.
\newblock {\em Progress in Academic Computational Integer Programming}, pages
  483--506.
\newblock Springer Berlin Heidelberg, Berlin, Heidelberg, 2013.

\bibitem{10.2307/1910129}
A.~H. Land and A.~G. Doig.
\newblock An automatic method of solving discrete programming problems.
\newblock {\em Econometrica}, 28(3):497--520, 1960.

\bibitem{ncsps_lhomme}
O.~Lhomme.
\newblock Consistency techniques for numeric csps.
\newblock In {\em Proceedings of the 13th International Joint Conference on
  Artifical Intelligence - Volume 1}, IJCAI'93, page 232–238, San Francisco,
  CA, USA, 1993. Morgan Kaufmann Publishers Inc.

\bibitem{MACKWORTH197799}
A.~K. Mackworth.
\newblock Consistency in networks of relations.
\newblock {\em Artificial Intelligence}, 8(1):99--118, 1977.

\bibitem{marriot}
K.~Marriott and P.~Stuckey.
\newblock {\em {Programming with Constraints: An Introduction}}.
\newblock The MIT Press, 02 1998.

\bibitem{masini}
R.~Mohr and G.~Masini.
\newblock Good old discrete relaxation.
\newblock In {\em Proceedings of the 8th European Conference on Artificial
  Intelligence}, ECAI'88, page 651–656, USA, 1988. Pitman Publishing, Inc.

\bibitem{NemhauserWolsey1988}
G.~Nemhauser and L.~Wolsey.
\newblock {\em Integer and Combinatorial Optimization}.
\newblock John Wiley \& Sons, Inc., 1988.

\bibitem{quarteroni2007numerical}
A.~Quarteroni, R.~Sacco, and F.~Saleri.
\newblock {\em Numerical mathematics}, volume~37.
\newblock Springer, 2007.

\bibitem{cp_handbook}
F.~Rossi, P.~van Beek, and T.~Walsh.
\newblock {\em Handbook of Constraint Programming}.
\newblock Elsevier Science Inc., USA, 2006.

\bibitem{Savelsbergh1994}
M.~W.~P. Savelsbergh.
\newblock Preprocessing and probing techniques for mixed integer programming
  problems.
\newblock {\em ORSA Journal on Computing}, 6:445--454, 1994.

\bibitem{eff_prop_engines}
C.~Schulte and P.~J. Stuckey.
\newblock Efficient constraint propagation engines.
\newblock {\em ACM Trans. Program. Lang. Syst.}, 31(1), Dec. 2008.

\bibitem{SofranacGleixnerPokutta2020}
B.~Sofranac, A.~Gleixner, and S.~Pokutta.
\newblock Accelerating domain propagation: An efficient gpu-parallel algorithm
  over sparse matrices.
\newblock In {\em 2020 IEEE/ACM 10th Workshop on Irregular Applications:
  Architectures and Algorithms (IA3)}, pages 1--11, 2020.

\bibitem{VANHENTENRYCK1998139}
P.~{Van Hentenryck}, V.~Saraswat, and Y.~Deville.
\newblock Design, implementation, and evaluation of the constraint language
  cc(fd).
\newblock {\em The Journal of Logic Programming}, 37(1):139--164, 1998.

\end{thebibliography}

\end{document}